\documentclass[twoside,11pt]{article}

\usepackage[preprint]{jmlr2e}

\usepackage[utf8]{inputenc} 
\usepackage[T1]{fontenc}    
\usepackage{hyperref}       
\usepackage{url}            
\usepackage{booktabs}       
\usepackage{amsfonts}       
\usepackage{nicefrac}       
\usepackage{microtype}      
\usepackage{subcaption}
\usepackage{graphicx}
\usepackage{float}
\usepackage{amsmath,amssymb}
\usepackage{mathptmx}      
\usepackage{times}
\usepackage{mathtools}
\usepackage{epstopdf}
\usepackage{natbib}
\usepackage{hyperref}
\hypersetup{
    colorlinks=true,
    linkcolor=blue,
    filecolor=blue,      
    urlcolor=blue,
    citecolor=blue,
}
\usepackage{bbm}
\usepackage{algorithmic}
\usepackage{algorithm}
\usepackage{color}
\usepackage[title]{appendix}
\usepackage{caption} 
\graphicspath{ {./images/} }

\newcommand{\Rb}{\mathbbm{R}}      

\newcommand{\Ac}{\mathcal{A}}

\newcommand{\Eb}{\mathbbm{E}}
\newcommand{\Fc}{\mathcal{F}}

\newcommand{\Kc}{\mathcal{K}}
\newcommand{\Lc}{\mathcal{L}}

\newcommand{\Sc}{\mathcal{S}}

\newcommand{\Zc}{\mathcal{Z}}
\newcommand{\1}{\mathbbm{1}}

\newcommand{\argmin}{\mathop{\rm argmin}}

\newcommand{\dist}{\mathop{\rm dist}}
\newcommand{\proj}{{\rm Proj}}

\newcommand*{\dt}[1]{\overset{\hbox{\tiny${\;\,}_\bullet$}}{#1}}

\newcommand{\Pbb}{\mathbb{P}}

\newenvironment{tightitemize}{%
    \list{{\textup{$\bullet$}}}{\settowidth\labelwidth{{\textup{\qquad}}}
    \leftmargin\labelwidth \advance\leftmargin\labelsep
    \parsep 0pt plus 1pt minus 1pt \topsep 3pt \itemsep 3pt
    }}{\endlist}

\newenvironment{tightlist}[1]{%
    \list{{\textup{(\roman{enumi})}}}{\settowidth\labelwidth{{\textup{(#1)}}}
    \leftmargin 0pt \advance\leftmargin\labelsep \itemindent \parindent
    \parsep 0pt plus 1pt minus 1pt \topsep 0pt \itemsep 0pt
    \usecounter{enumi}}}{\endlist}

\def\mg#1{\textcolor{black}{#1}}
\usepackage{soul}

\usepackage{soul}
\usepackage[normalem]{ulem}

\usepackage{todonotes}

\jmlrheading{1}{2020}{1-23}{06/08}{--/--}{tbd}{Mert G\"urb\"uzbalaban, Andrzej Ruszczy{\'n}ski and Landi Zhu}

\ShortHeadings{A Robust Stochastic Subgradient Method}{G\"urb\"uzbalaban, Ruszczy{\'n}ski and Zhu}
\firstpageno{1}

\begin{document}

\title{A Stochastic Subgradient Method for\\ Distributionally Robust Non-Convex Learning 
}

\author{\name Mert G\"{u}rb\"{u}zbalaban$^\text{*}$ \email mg1366@rutgers.edu 
\AND
\name Andrzej Ruszczy{\'n}ski$^\text{*}$ \email rusz@rutgers.edu
       \AND
       \name  Landi Zhu$^\text{*}$ \email lz401@scarletmail.rutgers.edu\\
      \addr Department of Management Science and Information Systems\\ Rutgers Business School, Piscataway, USA. \\
      *The authors are in alphabetical order. 
      }

\editor{TBD}

\maketitle

\begin{abstract} We consider a distributionally robust formulation of stochastic optimization problems arising in statistical learning, where robustness is with respect to uncertainty in the underlying data distribution. Our formulation builds on risk-averse optimization techniques and the theory of coherent risk measures. It uses semi-deviation risk for quantifying uncertainty, allowing us to compute solutions that are robust against perturbations in the population data distribution. { We consider a broad class of generalized differentiable loss functions that can be non-convex and non-smooth, involving upward and downward cusps, and we develop an efficient stochastic subgradient method for distributionally robust problems with such functions. We prove that it converges to a point satisfying the optimality conditions. To our knowledge, this is the first method with rigorous convergence guarantees in the context of generalized differentiable non-convex and non-smooth distributionally robust stochastic optimization. Our method allows for control of the desired level of robustness with little extra computational cost compared to population risk minimization \mg{with stochastic gradient methods}.} We also illustrate the performance of our algorithm on real datasets arising in convex and non-convex supervised learning problems. 
\end{abstract}

\section{Introduction}

Statistical learning theory deals with the problem of making predictions and constructing models from a set of data. A
typical statistical learning problem can be formulated as a stochastic optimization problem:
\begin{equation} \min_{x\in X} \mathbb{E}_{D\sim \mathbb{P}} \left[\ell(x,D)\right],
\label{pbm-stoc-opt}
\end{equation}
where $\ell:\Rb^n\times \Rb^d\to \Rb$ is the loss function of the predictor $x$ on the random data $D$ with an unknown distribution with probability law $\mathbb{P}$, and $X\subset\mathbb{R}^n$ is the feasible set (see, \emph{e.g.}, \citep{vapnik2013nature}). We consider loss functions that can be non-convex or non-differentiable (non-smooth). This framework includes a large class of problems in supervised learning including deep learning, linear and non-linear regression and classification tasks \citep{shalev2014understanding}.

A central problem in statistics is to make decisions that generalize well (i.e. work well on unseen data) as well as decisions that are robust to perturbations in the underlying data distribution \citep{DASZYKOWSKI2007203}. Indeed, the statistical properties of the input data may be subject to some variations and distributional shifts and a major goal is to build models that are not too sensitive to small changes in the input data distribution. This motivates the following distributionally robust version of the problem \eqref{pbm-stoc-opt}:
\begin{equation}
\min_{x\in X} \max_{\mathbb{Q}\in \mathcal{M}(\mathbb{P})} \mathbb{E}_{D\sim \mathbb{Q}} \left[\ell(x,D)\right],
\label{pbm-robust-stoc-opt}
\end{equation}
where $\mathcal{M}(\mathbb{P})$ is a weakly closed convex set of probability measures that models perturbations to the law $\mathbb{P}$, and the predictor $x$ is chosen to accommodate worst-case perturbations. 
References \citep{BAKER2008253,seidman2020robust,madry2017towards} provide thorough discussion of the relevance of robustness in
statistical learning. Problem \eqref{pbm-robust-stoc-opt} is related to quantifying risk of the random data distribution \citep{gao2017wasserstein,robust-cvar}; its computational tractability depends on the underlying risk measure and the uncertainty set $\mathcal{M}(\mathbb{P})$ \citep{ruszczynski2006,duchi2018learning,esfahani2018data}. Existing approaches to modelling $\mathcal{M}(\mathbb{P})$ include conditional value at risk \citep{robust-cvar}, $f$-divergence based sets \citep{duchi2018learning}, Wasserstein balls around $\mathbb{P}$ \citep{sinha2017certifying,gao2016distributionally}, and other statistical distance based approaches (see, \emph{e.g.}, \citep{gao2016distributionally}). When $\ell$ is non-convex and non-differentiable, these formulations lead to non-convex min-max problems. To our knowledge,  none of the existing algorithms admit provable convergence guarantees to a stationary point of \eqref{pbm-robust-stoc-opt} in this general case. 
\cite{sinha2017certifying} consider the case when $\mathcal{M}(\mathbb{P})$ is defined as the $\rho$-neighborhood of the probability law $\mathbb{P}$ under the
Wasserstein metric, where $\rho$ is the desired level of robustness. The authors formulate a Lagrangian relaxation of this problem for a fixed penalty parameter $\gamma\geq 0$ and show that when the loss is smooth and the penalty parameter is large enough (or by duality if the desired level of robustness $\rho$ is small enough), the stochastic gradient descent (SGD) method achieves the rates of convergence that are achievable in standard smooth non-convex optimization. The authors also provide a data-dependent upper bound for the worst-case population
objective \eqref{pbm-robust-stoc-opt} for any robustness level $\rho$. Soma and Yoshida \citep{soma2020statistical} proposed a conditional value-at risk (CVaR) formulation for robustness and show that for convex and smooth losses their algorithm based on SGD has $\mathcal{O}(1/\sqrt{n})$-convergence to the optimal CVaR, where $n$ is the number of samples. For nonconvex and smooth loss functions, they also show a generalization bound on the CVaR. However, none of these guarantees apply if the loss is non-smooth.

{For some structured regression and classification problems of practical interest,   distributionally robust formulations that result in finite-dimensional convex programs are known \citep{robust-log-reg,esfahani2018data,mehrotra,kuhn2019wasserstein} to be solvable in polynomial time; see also the reference \citep{tractable-review} which contains a detailed list of tractable reformulations of distributionally robust constraints for several risk measures. For convex losses, conic interior point solvers or gradient descent with backtracking Armijo line-searches 
can also be used for solving a sample-based approximation of \eqref{pbm-robust-stoc-opt}, when $\mathcal{M}(\mathbb{P})$ is defined via the $f$-divergences \citep{duchi2018learning}. However, these approaches can be prohibitively expensive when the dimension or the number of samples are large}. {For smooth and convex losses, Namkoong and Duchi \citep{namkoong2016stochastic} showed that a sample-based approximation of \eqref{pbm-robust-stoc-opt} with $f$-divergences results in a min-max problem which can then be solved with a bandit mirror descent algorithm with number of iterations comparable to that of the SGD for solving the sample-based approximation of the problem \eqref{pbm-stoc-opt}. However, similar convergence guarantees for non-convex or non-smooth losses were not given.} {We also note that} there are data-driven distributionally robust stochastic optimization formulations (see, \emph{e.g.}, \citep{esfahani2018data,gao2017wasserstein,gao2016distributionally}) {which} replace the population measure $\mathbb{P}$ with an empirical measure $\mathbb{P}^{N}$ constructed from samples of input data. A disadvantage is that the resulting set $\mathcal{M}(\mathbb{P}^{N})$ becomes random.

We propose a new formulation of \eqref{pbm-robust-stoc-opt} based on the 
mean--semideviation measure of risk \citep{OR:1999,OR:2001}. We propose a specialized stochastic subgradient method for solving the resulting
problem, {which we call the \emph{single-time scale} (STS) method.} Our method has local convergence guarantees for a large class of possibly non-convex and non-smooth loss functions. 

\textbf{Modeling $\mathcal{M}(\mathbb{P})$ with mean semi-deviation risk.} 
Consider the random loss $Z = \ell(x,D)$ 
 defined on a sample space $\varOmega$ equipped with a sigma algebra $\Fc$. We assume $\mathbb{E}(Z)$ to be finite, \emph{i.e.}, $Z \in \Zc = \Lc_1(\varOmega,\Fc,\mathbb{P})$. The mean--semideviation risk measure is defined as follows:
\begin{equation}
\label{msd}
\rho[Z]  = \Eb[Z] + \varkappa\, \Eb\big[\max\big(0,Z - \Eb[Z]\big)\big],\qquad \varkappa \in [0,1].
\end{equation}
It is known to be a coherent measure of risk \citep{ADEH:1999} (see also \citep{follmer2011stochastic,shapiro2009lectures} and the references therein). In particular, it has the \emph{dual representation} \citep{RuSh:2006a},
\[
\rho[Z] = \max_{\mu \in \Ac} \int_\varOmega Z(\omega)\mu(\omega)\;\mathbb{P}(d\omega) =
{\max_{\mathbb{Q}~:~ \frac{d\mathbb{Q}}{d\mathbb{P}}\in \Ac} \int_\varOmega Z(\omega)\;\mathbb{Q}(d\omega)}=
\max_{{\mathbb{Q} ~:~ \frac{d\mathbb{Q}}{d\mathbb{P}}\in \Ac}}\Eb_\mathbb{{Q}}[Z],
\]
where $\Ac$ is a convex and closed set defined as  follows:
\begin{equation}
\Ac = \big\{ \mu= \1+ \xi - \Eb[\xi]: \ \xi\in \Lc_{\infty}(\varOmega,\Fc,\mathbb{P}), \  \|\xi\|_\infty \le \varkappa,\ \xi \ge 0 \big\}.
\notag
\end{equation}
This provides \eqref{msd} with robustness with respect to the probability distribution; the level of robustness is controlled by the parameter $\varkappa$.
{After plugging $Z=\ell(x,D)$ into this formulation,
we obtain}
\begin{equation}
\label{stoch_prob}
\min_{x\in X} \max_{\mathbb{Q}\in\mathcal{M}(\mathbb{P})}\Eb_\mathbb{Q}[\ell(x,D)] = \min_{x\in X}\;  \Eb\Big[\ell(x,D) + \varkappa \max\big(0,\ell(x,D) - \Eb[\ell(x,D)]\big)\Big],
\end{equation}
with the perturbation set
\begin{equation}
\mathcal{M}(\mathbb{P}) = \big\{ \mathbb{Q} : \frac{d\mathbb{Q}}{d\mathbb{P}} \in \mathcal{A}\big\}.
\label{eq-Mp}
\end{equation}
An advantage of the formulation \eqref{stoch_prob} is that the perturbation set is implied rather than defined with the use of a metric in the space of probability measures.

Problem \eqref{stoch_prob} can be cast in the following form
of a composition optimization problem:
\begin{equation}
\label{main_prob}
\min_{x\in X}\; f(x,h(x)),
\end{equation}
with the functions
\begin{align}
f(x,u) &= \Eb\Big[\ell(x,D) + \varkappa\max\big(0,\ell(x,D) - u \big)\Big], \label{f-def}\\
h(x) &= \Eb[\ell(x,D)]. \label{h-def}
\end{align}
 The main difficulty is that neither  values nor (sub)gradients of $f(\cdot)$, $h(\cdot)$, and of their composition are available.
Instead, we postulate access to their random estimates. Such estimates, however, may be biased, because estimating a
(sub)gradient of the composition $F(x) = f(x,h(x))$ involves estimating $h(x)$. Although problem \eqref{main_prob} can be further rewritten in the standard format of composition optimization,
\begin{equation}
\label{composition}
\min_{x\in X} f(\bar{h}(x)),
\end{equation}
with $\bar{h}(x) = (x,h(x))$, but the more specific formulation \eqref{main_prob} allows us to derive a more efficient specialized method, because $x$ is observed.

%

 The research on composition optimization problems of form \eqref{composition}  started from
penalty functions for stochastic constraints and composite regression models in \citep[Ch. V.4]{ermoliev1976methods}.
An established approach was to use two-level stochastic
recursive algorithms with two stepsize sequences in different time scales:  a slower one for updating the
main decision variable $x$, and a faster one for filtering the value of the inner function $h$.
References \citep{wang2017stochastic,WaLiFa17,kalogerias2018recursive,yang2019multi} provide a detailed account of these techniques and existing results.

A Central Limit Theorem for stochastic versions of  problem \eqref{composition} has been established in  \citep{dentcheva2017statistical}.  Large deviation bounds for the empirical optimal value were derived in \citep{ermoliev2013sample}. A new single time-scale method for problem \eqref{composition} with continuously differentiable functions has been recently proposed in \citep{ghadimi2018single}. It has  the complexity of ${\cal O}(1/\epsilon^2)$ to obtain an $\varepsilon$-solution
 of the problem, the same as methods for one-level unconstrained stochastic optimization. However, the construction of the method
 and its analysis depend on the Lipschitz constants of the gradients of the functions involved. Our problem \eqref{main_prob}, unfortunately, involves a nonsmooth function $\max(\cdot,\cdot)$, and may also
 involve a nonsmooth (non-differentiable) loss function $\ell(\cdot,\cdot)$. Indeed, many key problems in machine learning involve non-convex and non-smooth loss functions. A prominent example is deep learning with ReLU activation functions (see \emph{e.g.} \citep{goodfellow2016deep}). There are many other statistical learning problems where the objective can be non-smooth and non-differentiable such as non-convex generalized linear models and non-convex regression and risk minimization (see \emph{e.g.} \citep{hastie2015statistical,nonconvex-learning,allen2016variance,JMLR:v11:teo10a}). The organic non-differentiability and non-convexity are additional challenges for the solution method.

\textbf{Contributions.} We propose to model the perturbation to input data distribution by mean-semi\-deviation risk, according to \eqref{eq-Mp}. Our formulation leads to the distributionally robust learning problem \eqref{stoch_prob} which has the advantage that it results in a convex optimization problem when the loss $\ell$ is convex, {in contrast to some alternative formulations which result in min-max optimization problems (see, \emph{e.g.}, \citep{robust-cvar,namkoong2016stochastic}}).
When the loss is non-convex and non-smooth, we can still find a stationary point to \eqref{stoch_prob}, by our novel single time-scale parameter-free stochastic subgradient method. We prove that it finds a stationary point of \eqref{stoch_prob} for a general class of loss functions that can be non-convex and non-differentiable. 
To our knowledge, out method is the first method with {probability one} convergence guarantees for solving a distributionally robust formulation of a population minimization problem, where the loss can be non-convex or non-differentiable.

{We also note that the computational cost of stochastic first-order optimization algorithms are typically measured in terms of the number of stochastic gradient or subgradient evaluations they require (see, \emph{e.g.}, \citep[Section 6]{bubeck}, \citep{jain2018accelerating,ghadimi2013optimal}). Standard SGD methods (which go back to Robbins and Monro's pioneering work \citep{robbins1951stochastic}) applied to the non-robust optimization problem \eqref{pbm-stoc-opt} can operate with one stochastic subgradient evaluation under similar assumptions to ours, however they are not applicable to the robust formulation \eqref{pbm-robust-stoc-opt} directly. In contrast, our method can converge to a stationary point of the robust formulation \eqref{pbm-robust-stoc-opt} with probability one requiring at most two stochastic subgradient evaluations at every iteration. Therefore, comparing the numbers of stochastic gradient evaluations, the computational cost of each iteration of our method is at most twice that of the standard SGD method, requiring little extra computational cost for computing robust solutions.} 


\section{The single time-scale (STS) method with subgradient averaging}
\label{s:3}
{
We present the method for problems of the form \eqref{stoch_prob}, in which the loss function $\ell(x,D)$
is differentiable in a generalized sense \citep{norkin1980generalized} with respect to $x$ and integrable with respect to $D$.
This broad class of functions is contained in the set of locally Lipschitz functions,
and contains all semismooth locally Lipschitz loss functions that can be non-convex and non-differentiable \citep{mifflin1977semismooth}. 
{We note that this class includes many of the losses arising in statistical learning problems, including population and empirical risk minimization with possibly non-convex and non-smooth regularizers \citep{hastie2015statistical,nonconvex-learning,allen2016variance,JMLR:v11:teo10a,vapnik2013nature}, weakly convex and continuous losses \citep{davis-dima,lee-weakly-convex} as well as deep learning with ReLU activations \citep{goodfellow2016deep}.}

Recall that the Clarke subdifferential $\partial_x\ell(x,D)$ is an inclusion-minimal generalized derivative of $\ell(\cdot,D)$ \citep{norkin1980generalized}.
 We make the following assumptions.

\begin{tightitemize}
\item[(A1)] The set $X\subset \Rb^n$ is convex and compact;
\item[(A2)] For almost every (a.e.) $\omega\in \Omega$, the function $\ell(\cdot,D(\omega))$ is differentiable in a generalized sense with
the subdifferential $\partial_x \ell(x,D(\omega))$, $x\in \Rb^n$.  Moreover, for every compact set $K \in \Rb^n$ an integrable function $L_K:
\varOmega\to \Rb$ exists, satisfying
$\sup_{x\in K}\sup_{g \in \partial \ell(x,D(\omega)) } \|g\| \le L_K(\omega)$.
\end{tightitemize}
Under (A2), the function \eqref{h-def} is also differentiable in a generalized sense. Although its generalized derivative is not readily available, we can draw $\widetilde{D}$ from the distribution of $D$
and use an element of $\partial_x\ell(x,\widetilde{D})$ as a \emph{stochastic subgradient}  (a random vector whose expected value is a subgradient).
Furthermore, the function \eqref{f-def} is also differentiable in a generalized sense with respect to $(x,u)$.
Its stochastic subgradient can be obtained as follows. First, we observe
$\ell(x,\widetilde{D})$ and choose
\[
\lambda \in  \begin{cases} \{0\} & \text{ if } \ell(x,\widetilde{D}) < u,\\
 [0,1] & \text{ if } \ell(x,\widetilde{D}) = u,\\
 \{1\} & \text{ if } \ell(x,\widetilde{D}) > u.
\end{cases}
\]
Then the vector $\begin{bmatrix} \tilde{g}_x \\ \tilde{g}_u \end{bmatrix}$, where
$\tilde{g}_x \in (1+ \lambda \varkappa)\partial_x\ell(x,\widetilde{D})$, $\tilde{g}_u = -\lambda\varkappa$,
is a stochastic subgradient of the function $f(x,u)$ which is defined by \eqref{f-def}. These formulas follow from calculus rules for generalized subdifferentials
of compositions \citep[Thm. 1.6]{mikhalevich1987nonconvex} and expected values \citep[Thm. 23.1]{mikhalevich1987nonconvex}.
We can also use different samples for calculating  stochastic subgradients of \eqref{f-def} and \eqref{h-def}.
}

The {STS} method generates three random sequences: approximate solutions $\{x^k\}$,
path-averaged stochastic subgradients $\{z^k\}$, and inner function estimates $\{u^k\}$, all defined
on a certain probability space $(\Omega,\Fc,P)$. We let $\Fc_k$ to be
the $\sigma$-algebra generated by
$\{x^0,\dots,x^k,z^0,\dots,z^k,u^0,\dots,u^k\}$.
Starting from the initialization
$x^0 \in X$, $z^0 \in \Rb^n$, $u^0 \in \Rb$, the method uses parameters  $a>0$, $b>0$ and ${c}>0$ to generate $x^k, z^k, u^k$ for $k>0$. 
At each iteration $k=0,1,2,\dots$, we compute
\begin{equation}
\label{QP}
y^k = \argmin_{y \in X}\  \left\{\langle z^k, y-x^k \rangle + \frac{{c}}{2} \|y-x^k\|^2\right\},
\end{equation}
and, with an $\Fc_k$-measurable stepsize $\tau_k \in \big(0,\min(1,1/a)\big]$, we set
\begin{equation}
\label{def_xk}
x^{k+1} = x^k + \tau_k (y^k-x^k).
\end{equation}
Then, we obtain statistical estimates:
\begin{tightitemize}
\item $\tilde{g}^{k+1}=\begin{bmatrix}\tilde{g}_x^{k+1} \\ \tilde{g}_u^{k+1}\end{bmatrix}$ of an element ${g}^{k+1}=\begin{bmatrix} g_x^{k+1}\\ g_u^{k+1}\end{bmatrix}\in \partial\!f(x^{k+1},u^k)$,
\item $\tilde{h}^{k+1}$  of $h(x^{k+1})$, and
\item $\tilde{J}^{\,k+1}$ of an element $J^{k+1}\in \partial h(x^{k+1})$ with the convention that $J^{k+1}$ is a row vector,
\end{tightitemize}
and we update the running averages as
\begin{align}
z^{k+1} &= z^k + a\tau_k \Big(\tilde{g}_x^{k+1}+\big[\tilde{J}^{\,k+1}\big]^\top \tilde{g}_u^{k+1} - z^k\Big),  \label{def_zk}\\
u^{k+1} &= u^k + \tau_k \tilde{J}^{\,k+1} (y^k-x^k) + b \tau_k \big(\tilde{h}^{k+1}-u^k\big).   \label{def_uk}
\end{align}
We assume the following conditions on the stepsizes and the stochastic estimates:
\begin{tightitemize}
\item[(A3)] $\tau_k \in \big(0,\min(1,1/a)\big]$ for all $k$, $\lim_{k\to\infty}\tau_k= 0$, $\sum_{k=0}^\infty \tau_k = \infty$,
$\sum_{k=0}^\infty \Eb[\tau_k^2] < \infty$;
\item[(A4)] For all $k$,
\begin{tightlist}{iii}
\item
$\tilde{g}^{k+1} = g^{k+1} + e_g^{k+1} + \delta_g^{k+1}$, with \\ $g^{k+1}\in \partial\!f(x^{k+1},u^k)$,
 $\Eb\big\{e_g^{k+1}\big|\Fc_k\big\} = 0$, $\Eb\big\{\|e_g^{k+1}\|^2|\Fc_k\big\}\le \sigma_g^2$, \\$\lim_{k\to \infty} \delta_g^{k+1}=0$,
\item $\tilde{h}^{k+1} = h(x^{k+1}) + e_h^{k+1} + \delta_h^{k+1}$, with\\
$\Eb\big\{e_h^{k+1}\big|\Fc_k\big\} = 0$,  $\Eb\big\{\|e_h^{k+1}\|^2|\Fc_k\big\}\le \sigma_h^2$, $\lim_{k\to \infty} \delta_h^{k+1}=0$,
\item $\tilde{J}^{\,k+1} = J^{k+1} + E^{k+1} + \Delta^{k+1}$,with \\
$J^{k+1}\in \partial h(x^{k+1})$,
$\Eb\big\{E^{k+1}\big|\Fc_k\big\} = 0$, $\Eb\big\{\|E^{k+1}\|^2|\Fc_k\big\}\le \sigma_E^2$,
\end{tightlist}
and $e_g^{k+1}$ and $E^{k+1}$ are statistically independent, given $\Fc_k$.
\end{tightitemize}
These assumptions are pretty standard in the study of stochastic gradient and stochastic approximation methods \citep{kushner2003stochastic}. As discussed before, the stochastic estimates satisfying these conditions can be obtained by drawing at each iteration one or two independent samples: $D_1^{k+1}$ and $D_2^{k+1}$, from the data. Then we can take
\begin{align*}
\tilde{g}_x^{k+1} &\in \begin{cases} \partial_x \ell(x^{k+1}, D_1^{k+1}) & \text{if } \ell(x^{k+1}, D_1^{k+1}) < u^k,\\
                                    (1+\varkappa) \partial_x \ell(x^{k+1}, D_1^{k+1}) & \text{if } \ell(x^{k+1}, D_1^{k+1}) \ge u^k,
                                    \end{cases}\\
\tilde{g}_u^{k+1} &=\begin{cases} 0 & \text{if } \ell(x^{k+1}, D_1^{k+1}) < u^k,\\
                                   -\varkappa & \text{if } \ell(x^{k+1}, D_1^{k+1}) \ge u^k,
                                    \end{cases}\\
\tilde{h}^{k+1} &=  \ell(x^{k+1}, D_1^{k+1}),
\\
\tilde{J}^{\,k+1} &\in \begin{cases}
\big\{\tilde{g}_x^{k+1}\big\}  & \text{if } \ell(x^{k+1}, D_1^{k+1}) < u^k,\\
 \partial_x \ell(x^{k+1}, D_2^{k+1}) & \text{if } \ell(x^{k+1}, D_1^{k+1}) \ge u^k.
\end{cases}
\end{align*}
When $\ell(x^{k+1}, D_1^{k+1}) < u^k$, only one sample, $D_1^{k+1}$, is needed, because
$\big[\tilde{J}^{\,k+1}\big]^\top \tilde{g}_u^{k+1}=0$ in \eqref{def_zk} in this case. {In any case,  each iteration of our algorithm requires at most two stochastic subgradient evaluations.}

Our method refines and specializes the approach to multi-level stochastic optimization recently developed in \citep{ruszczynski2020stochastic}. { We extend this approach to a new case in which the upper level function function is not continuously differentiable and thus the
conditions of \citep{ruszczynski2020stochastic} are not satisfied. We establish the convergence in the new case as well, as detailed in the following section.}

\section{Convergence analysis}

To recall optimality conditions for problem \eqref{stoch_prob}, and analyze our method, we need to introduce  relevant multifunctions.
For a point $x\in \Rb^n$, we define the set:
\begin{equation}
\label{GF}
G_F(x) = {\rm conv} \big\{ s\in \Rb^{n}: s = g_x + J^\top g_u,\ g \in \partial\! f(x,h(x)), \ J \in \partial h(x) \big\}.
\end{equation}
By \citep[Thm. 1.6]{mikhalevich1987nonconvex}, 
the set $G_F(x)$ is a generalized subdifferential of the composition function $F(x)=f(x,h(x))$.
We call a point $x^*\in X$ \emph{stationary} for problem \eqref{stoch_prob}, if
\begin{equation}
\label{stationary}
0 \in  G_F(x^*) + N_X(x^*),
\end{equation}
where $N_X(x)$ is the normal cone to $X$ at $x$.
The set of stationary points
is denoted by $X^*$. We start from a useful property of the gap function {$\eta:X\times \Rb^n\to(-\infty,0]$},
\begin{equation}
\label{gap}
\eta(x,z) = \min_{y\in X} \left\{\langle z, y-x \rangle + \frac{{c}}{2} \|y-x\|^2\right\}.
\end{equation}
We denote the minimizer in \eqref{gap} by $\bar{y}(x,z)$. Since it is a projection of $x-z/{c}$ on $X$,
\begin{equation}
\label{opt-eta}
 \langle z ,\bar{y}(x,z)-x \rangle  + {c} \| \bar{y}(x,z)-x \|^2 \le 0.
\end{equation}
Moreover, a point $x^*\in X^*$ if and only if $z^*\in G_F(x^*)$ exists
 such that $\eta(x^*,z^*)=0$. Consider the multifunction $\varGamma: \Rb^n\times \Rb^n \times \Rb \rightrightarrows \Rb^n \times \Rb$:
\begin{equation}
\label{Gamma-def}
\begin{aligned}
\varGamma(x,z,u)=&\big\{(R,v): \exists g \in \partial f(x,u), \exists J \in \partial h(x), \\
&v=J  \big(\bar{y}(x,z)-x\big) + b(h(x)-u), \ R=a\big(g_{x}+J^\top g_{u}-z\big)\big\}.
\end{aligned}
\end{equation}

With this notation, we can write the updates \eqref{def_zk}--\eqref{def_uk} as follows:
\begin{equation}
\label{u-abstract}
\begin{bmatrix}
z^{k+1} \\ u^{k+1}
\end{bmatrix} \in  \begin{bmatrix} z^k \\ u^k \end{bmatrix}
+ \tau_k \varGamma(x^{k+1},z^k,u^k) + \tau_k \theta^{k+1} + \tau_k \alpha^{k+1},
\end{equation}
where, for some constant $C^{\theta}$,
\begin{equation}
\label{theta-abstract}
\Eb\big[\theta^{k+1}\,\big|\, \Fc_k\big]=0,  \quad \Eb\big[\|\theta^{k+1}\|^2\,\big|\,\Fc_k\big]\le C^{\theta}, \quad  k=0,1,\dots
\end{equation}
and
\begin{equation}
\label{alpha-abstract}
\lim_{k\to \infty} \alpha^{k+1} = 0.\quad \text{a.s.}.
\end{equation}
The verification of relations \eqref{u-abstract}--\eqref{alpha-abstract} is  straightforward from the
description of the algorithm and assumptions (A3)--(A4). Two technical results are needed for further analysis.

\begin{lemma}
\label{l:Gamma}
The multifunction $\varGamma$
is compact and convex valued.
\end{lemma}
\begin{proof}
By assumption, for a.e. $\omega \in \Omega$, the loss function $\ell(x,D(\omega))$ is generalized differentiable, and therefore the function $f(x,u)$ is also generalized differentiable where $\partial_x f(x,u)$, $\partial_u f(x,u)$ and $\partial h(x)$ are all convex and compact \citep{norkin1980generalized}.

Since the function $F(u,D(\omega))=\ell(x,D(\omega))+\varkappa \cdot \max(0,\ell(x,D(\omega)-u)$ {is generalized differentiable for a.e. $\omega$}, 
{by the interchangeability of the generalized subdifferential and integral operators \citep[Thm. 23.1]{mikhalevich1987nonconvex}, we obtain}:
\begin{equation}
\partial_u f (x,u)=\mathbb{E} [\partial_u F(u,D)].
\end{equation}
We also have
$$ \partial_uF(u,D)=\varkappa \cdot \left\{
\begin{aligned}
-1, &\ u<\ell(x,D), \\
[-1,0],&\ u=\ell(x,D), \\
0,&\ u>\ell(x,D),
\end{aligned}
\right.
$$
which implies
\begin{align*}
 \partial_u f(x,u)&= - \varkappa\, \Pbb\big\{u<\ell(x,D)\big\}+\varkappa \big[-\Pbb\big\{u=\ell(x,D)\big\},0\big]\\
&=\varkappa \big[-\Pbb\big\{u \le \ell(x,D)\big\}, -\Pbb\big\{u<\ell(x,D)\big\}\big].
\end{align*}

If we denote $\Pbb\big\{u\le \ell(x,D)\big\}$ and $\Pbb\big\{u<\ell(x,D)\big\}$ by $P_1$ and $P_2$ respectively, we obtain
$\partial_u f(x,u) = \varkappa \cdot [-P_1,-P_2]$.

Now, in order to prove that $\varGamma (x,z,u)$ is convex-valued, we choose two points in $\varGamma (x,z,u)$: $A=(R_a,v_{1a},v_{2a})$ and $B=(R_b,v_{1b},v_{2b})$. Since every point in $\varGamma (x,z,u)$ is generated by  a pair of $(g,J)$ from $\partial f(x,u) \times \partial h(x) $, we can also denote the pair generating the point $A$ by $(a_1,a_2)$, and the pair generating the point $B$ by $(b_1, b_2)$.

For every $\theta \in [0,1]$, the convex combination $(R^\theta, v_1^\theta, v_2^\theta)$ of $A$ and $B$ can be expressed as:
\begin{align*}
R^\theta &= a(\theta a_{1x} + (1-\theta)b_{1x} + \theta a_{1u}a_2 +(1-\theta)b_{1u}b_2-z),\\
v_1^\theta &= (\theta a_{1x} + (1-\theta)b_{1x})(\bar{y}(x,z)-x)+\theta a_{1u} a_2 +(1-\theta)b_{1u}b_2 + b(f(x,u_2)-u_1),\\
v_2^\theta &= \theta a_2 +(1-\theta)b_2.
\end{align*}
If we can always find a pair $(c_1,c_2) \in \partial f(x,u) \times \partial h(x)$ that generates this convex combination, then $\varGamma (x,z,u)$ is convex-valued.

First, since $\partial_x f(x,u)$ and $\partial h(x)$ are convex sets, we can choose $c_{1x}=\theta a_{1x} + (1-\theta)b_{1x}, c_2=\theta a_2 + (1-\theta)b_2$ (we do not choose $c_{1u}$ yet); then the corresponding point $C=(g_c, v_{1c}, v_{2c})$ is:
\begin{align*}
R_c \  &= a(\theta a_{1x} + (1-\theta)b_{1x} + c_{1u} (\theta a_2 +(1-\theta)b_2)-z),\\
v_{1c} &= (\theta a_{1x} + (1-\theta)b_{1x})(\bar{y}(x,z)-x)+c_{1u} (\theta a_2 +(1-\theta)b_2) + b(f(x,u_2)-u_1),\\
v_{2c} &= \theta a_2 +(1-\theta)b_2.
\end{align*}
Furthermore, we have $c_{1u} \in \partial_u f(x,u)=\varkappa \cdot [-P_1,-P_2]$, so for the common item $c_{1u} (\theta a_2 +(1-\theta)b_2)$ in $R_c$ and $v_{1c}$, any value between $-\varkappa P_1 (\theta a_2 +(1-\theta)b_2)$ and $-\varkappa P_2 (\theta a_2 +(1-\theta)b_2)$ can be achieved.

On the other hand, for the common item $\theta a_{1u}a_2 +(1-\theta)b_{1u}b_2$ in $R^\theta$ and $v_1^\theta$, since $a_{1u},b_{1u} \in \partial_u f(x,u)=\varkappa \cdot [-P_1,-P_2]$, we have:
\begin{equation}
-\varkappa P_1 (\theta a_2 +(1-\theta)b_2) \le \theta a_{1u}a_2 +(1-\theta)b_{1u}b_2 \le -\varkappa P_2 (\theta a_2 +(1-\theta)b_2),
\end{equation}
so there must exist $c_{1u}^* \in \partial_u f(x,u)$ that satisfies:
\begin{align*}
c_{1u}^* (\theta a_2 +(1-\theta)b_2) &= \theta a_{1u}a_2 +(1-\theta)b_{1u}b_2.
\end{align*}
This implies that for this value of $c_{1u}$,
\begin{align*}
(R^\theta, v_1^\theta, v_2^\theta) = (R_c, v_{1c}, v_{2c}),
\end{align*}
and we conclude that $\varGamma (x,z,u)$ is convex-valued. Furthermore, because $\partial_x f(x,u), \partial_u f(x,u)$ and $\partial h(x)$ are all compact, the set $\varGamma (x,z,u)$ is compact as well.
\end{proof}

\begin{lemma}
\label{l:bounded}
The sequences $\{z^k\}$ and $\{u^k\}$ are bounded with probability 1.
\end{lemma}
The proof is routine and is therefore omitted.

We analyze the method by the differential inclusion technique, by refining and specializing the approach adopted in \citep{ruszczynski2020stochastic}.
{ Although our model does not fit the assumptions of \citep{ruszczynski2020stochastic},  our result on the convexity of the multifunction $\varGamma(\cdot)$ allows for proving convergence in this case as well.}

\begin{theorem}
\label{t:convergence}
If the assumptions {\rm (A1)--(A4)} are satisfied,
then with probability 1 every accumulation point $\hat{x}$ of the sequence $\{x^k\}$ is stationary, $\lim_{k\to\infty} (u^k-h(x^k))=0$, and the sequence $\{F(x^k)\}$ is convergent.
\end{theorem}
\begin{proof}
We consider a specific trajectory of the method and divide the proof into three standard steps.

\emph{Step 1: The Limiting Dynamical System.}  We denote by $p^k=(x^k,z^k,u^k)$, $k=0,1,2,\dots$, a realization
of the sequence generated by the algorithm.
We introduce the accumulated stepsizes
$t_k = \sum_{j=0}^{k-1}\tau_j$, $k=0,1,2 \dots$, and we construct the interpolated trajectory
\[
P_0(t) = p^k + \frac{t-t_{k}}{\tau_k}(p^{k+1}-p^k),\quad t_{k}\le t \le t_{k+1},\quad k=0,1,2,\dots.
\]
For an increasing sequence of positive numbers $\{s_k\}$ diverging to infinity, we define shifted trajectories $P_k(t) =P_0(t+s_k)$.
The sequence $\{p^k\}$ is bounded by Lemma \ref{l:bounded} and so are the functions $P_k(\cdot)$.

By \citep[Thm. 3.2]{majewski2018analysis},  for any infinite set $\Kc$ of positive integers,
there exist an infinite  subset $\Kc_1 \subset \Kc$ and an absolutely continuous function  $P_\infty:
[0,+\infty) \to X\times \Rb^n\times \Rb^m$ such that for any $T > 0$
\[
\lim_{\substack{{k\to\infty}\\{k\in \Kc_1}}}\sup_{t\in[0,T]} \big\| P_k(t)-P_\infty(t)\big\|= 0,
\]
and $P_\infty(\cdot)=\big(X_\infty(\cdot),Z_\infty(\cdot),U_\infty(\cdot)\big)$
is a solution of the system of differential equations and inclusions corresponding to \eqref{def_xk} and
 and \eqref{u-abstract}:
\begin{gather}
\dt{x}(t) = \bar{y}\big(x(t),z(t)\big)-x(t), \label{dx}\\
\big(\dt{z}(t), \dt{u}(t)\big) \in \varGamma(x(t),z(t),u(t)). \label{dz}
\end{gather}
Moreover, for any $t\ge 0$, the triple $\big(X_\infty(t),Z_\infty(t),U_\infty(t)\big)$
is an accumulation point of the sequence $\{(x^k,z^k,u^k)\}$.

In order to analyze the equilibrium points of the system \eqref{dx}--\eqref{dz}, we first study the dynamics
of the functions $H(t) = h(X(t))$ and $F(t) = f(X(t),U(t))$.
It follows from \eqref{dx} that the path $X(\cdot)$ is continuously differentiable. By virtue of assumption (A2)
and \citep[Thm. 1]{Rusz-OptLet},
for any $J(t) \in \partial h(X(t))$,
\begin{equation}
\label{f-incr-M}
 \dt H(t)
=   J(t) \dt X(t).
\end{equation}
Again, Assumption (A2) and  \citep[Thm. 1]{Rusz-OptLet} imply that for any $G(t) \in \partial f(X(t),U(t))$,
\begin{equation}
\label{f-incr-m-2}
 \dt F(t)
=   G_x(t)^\top  \dt X(t)  + G_u(t)^\top  \dt U(t).
\end{equation}
To understand the dynamics of $U(\cdot)$, from \eqref{dz} and \eqref{u-abstract} we deduce that
\begin{equation}
\label{UM}
\dt U(t) = \hat{J}(t) \dt X(t) + b [H(t) - U(t)],
\end{equation}
with some $\hat{J}(t) \in \partial h(X(t))$.
Therefore, using $J(\cdot) = \hat{J}(\cdot)$ in \eqref{f-incr-M}, we obtain
\begin{equation}
\label{Um}
\dt U(t) =  \dt H(t) + b[ H(t) - U(t)].
\end{equation}
Consequently, the solution of \eqref{f-incr-m-2}--\eqref{UM} has the form:
 \begin{equation}
 \label{Phi-sol}
   \dt F(t) =
   \hat{G}(t)^\top \dt X(t) + b G_u(t)^\top[ H(t) - U(t)].
 \end{equation}
with $\hat{G}(t) = {G}_x(t) + \hat{J}(t)^\top  G_u(t)$.
These observations will help us study the stability of the system.

\emph{Step 2: Descent Along a Path.}  We use the Lyapunov function
\begin{equation}
\label{Lyapunov}
W(x,z,u) = a f(x,u) -  \eta(x,z)  + \gamma \big\|h(x)-u\big\|,
\end{equation}
with the coefficient $\gamma>0$ to be specified later.

Directly from \eqref{Phi-sol} we obtain
\begin{equation}
\label{f-incr2}
  f(X(T),U(T)) - f(X(0),U(0))
 = \int_0^T  \hat{G}(t)^\top \dt{X}(t)  \;dt
  + b  \int_0^T  G_{u}(t)^\top \big[H(t)-U(t)\big]\;dt.
\end{equation}

We now estimate the change of $\eta(X(\cdot),Z(\cdot))$ from 0 to $T$.
 Since $\bar{y}(x,z)$ is unique, the function $\eta(\cdot,\cdot)$ is continuously differentiable.
Therefore, the chain formula holds for it as well:
\begin{multline*}
\eta(X(T),Z(T)) - \eta(X(0),Z(0)) \\
= \int_0^T  \big\langle \nabla_x \eta(X(t),Z(t)), \dt{X}(t) \big\rangle \;dt +
\int_0^T  \big\langle \nabla_z \eta(X(t),Z(t)), \dt{Z}(t) \big\rangle \;dt.
\end{multline*}
From \eqref{dz} we obtain
\[
\dt{Z}(t) = a\big(\hat{G}^\top (t) - Z(t)\big),
\]
with the same $\hat{G}(\cdot)$ as in \eqref{Phi-sol} and \eqref{f-incr2}.

Substituting
$\nabla_x \eta(x,z) = -z+{c}(x- \bar{y}(x,z))$,
$\nabla_z \eta(x,z) =  \bar{y}(x,z)-x$,  and
using \eqref{opt-eta}, we obtain
\begin{align*}
\lefteqn{\eta(X(T),Z(T)) - \eta(X(0),Z(0))}\quad  \\
&= \int_0^T  \big\langle -Z(t)+{c}(X(t)-\bar y(X(t),Z(t)))\, , \,\bar{y}(X(t),Z(t)) - X(t) \big\rangle \;dt \\
&{\quad } +
a \int_0^T  \big\langle \bar y(X(t),Z(t))-X(t)\,,\,\hat{G}(t) - Z(t) \big\rangle \;dt\\
&\ge\; a \int_0^T  \big\langle \bar y(X(t),Z(t))-X(t)\,,\, \hat{G}(t) - Z(t) \big\rangle \;dt\\
&\ge\; a \int_0^T  \hat{G}(t)^\top\big( \bar y(X(t),Z(t))-X(t)\big) \;dt
 + a{c} \int_0^T  \big\| \bar y(X(t),Z(t))-X(t)\big\|^2 \;dt.
   \end{align*}
With a view at \eqref{dx}, we conclude that
\begin{equation}
\label{eta-incr}
 \eta(X(T),Z(T)) - \eta(X(0),Z(0))
  \ge \;  a \int_0^T  \hat{G}^\top(t) \dt X(t) \;dt
 + a{c} \int_0^T  \big\| \dt X(t)\big\|^2 \;dt.
  \end{equation}

We now estimate the increment of $\big\|  H(\cdot)-U(\cdot)\big\|$ from 0 to $T$. As
 $\| \cdot\| $ is convex and $H(\cdot)$ and $U(\cdot)$ are absolutely continuous, the chain rule applies as well: for any
 $\lambda(t) \in \partial \| H(t)- U(t)\|$ we have
\[
 \big\|H(T)-U(T)\big\| -  \big\|H(0)-U(0)\big\|
= \int_0^T \big\langle \lambda(t),  \dt H(t) - \dt U(t) \big\rangle\; dt.
\]
By \eqref{Um},  $\dt H(t) - \dt U(t) = b\big[ U(t)-H(t)\big]$ for almost all $t$.
Furthermore,
\[
\lambda_m(t)=\frac{H(t)-U(t)}{\|H(t)-U(t)\|},\quad \text{if}\quad  H(t)\ne U(t).
\]
Therefore
\begin{equation}
\label{norm-incr}
 \big\|H(T)-U(T)\big\| -  \big\|H(0)-U(0)\big\| = - b \int_0^T \big\| H(t)-U(t)\big\| \; dt.
 \end{equation}

We can now combine \eqref{f-incr2}, \eqref{eta-incr}, and \eqref{norm-incr} to estimate the change of the function \eqref{Lyapunov}:
\begin{multline*}
W\big(X(T),Z(T),U(T)\big) - W\big(X(0),Z(0),U(0)\big)\quad \\
\quad \le   a  b   \int_0^T  {G}_{u}(t)^\top\big[H(t)-U(t)\big]\;dt
 - a{c} \int_0^T  \big\| \dt X(t)\big\|^2 \;dt
- b \gamma  \int_0^T \big\|H(t)-U(t)\big\|\;dt.
\end{multline*}
Because the paths $X(t)$ and $U(\cdot)$ are bounded a.s. and the functions $f_m$ are locally Lipschitz,  a (random) constant $L$ exists, such that
$\big \| {G}_{u}(t)\big\| \le L$.
The last estimate entails:
\begin{multline}
\label{W-incr}
W\big(X(T),Z(T),U(T)\big) - W\big(X(0),Z(0),U(0)\big) \\
 \le - a{c} \int_0^T  \big\| \dt X(t)\big\|^2 \;dt - b (\gamma-aL)  \int_0^T \|H(t)-U(t)\|\;dt.
\end{multline}
By choosing $\gamma > aL$, we ensure that $W(\cdot)$ has the descent property to be used in our stability analysis at Step 3.
The fact that $L$ (and thus $\gamma$) may be different for different paths is irrelevant, because our analysis is path-wise.

\emph{Step 3: Analysis of the Limit Points.} Define the set
\[
\quad \Sc = \big\{ (x,z,u)\in X^* \times \Rb^n\times \Rb: \eta(x,z)=0,\
  u=h(x)\big\}.\quad
\]
Suppose $(\bar{x},\bar{z},\bar{u})$ is an accumulation point of the sequence $\{(x^k,z^k,u^k)\}$. If $\eta(\bar{x},\bar{z}) <0$ or
$\bar{u}\ne h(\bar{x})$, then
every solution $(X(t),Z(t),U(t))$ of the system \eqref{dx}--\eqref{dz},
starting from $(\bar{x},\bar{z},\bar{u})$ has $\|\dt X(0)\| > 0$ or $\|H(0)-U(0)\| >0$.
Using \eqref{W-incr} and arguing as in \citep[Thm. 3.20]{duchi2018stochastic} or \citep[Thm. 3.5]{majewski2018analysis},
we obtain a contradiction. Therefore,
we must have $\eta(\bar{x},\bar{z})=0$ and
$\bar{u}= h(\bar{x})$. Suppose $\bar{x}\not \in X^*$. Then
\begin{equation}
\label{non-opt}
\dist\big(0,  G_F(\bar{x}) + N_X(\bar{x})\big) >0.
\end{equation}
Suppose the system \eqref{dx}--\eqref{dz}
starts from $(\bar{x},\bar{z},\bar{u})$ and $X(t)=\bar{x}$ for all $t\ge 0$. From \eqref{dz} and \eqref{Gamma-def}, in view of the
equations $\bar{y}(\bar{x},\bar{z})=\bar{x}$ and  $\bar{u}=h(\bar{x})$, we obtain  $U(t)=  f(\bar{x})$ for all $t \ge 0$.
The inclusion \eqref{dz}, in view of \eqref{GF}, simplifies
\[
\dt{z}(t) \in a \big( G_F(\bar{x}) - z(t)\big).
\]
For the convex Lyapunov function $V(z) = \dist\big(z,G_F(\bar{x})\big)$,
we apply the classical chain formula \citep{brezis1971monotonicity}
on the path $Z(\cdot)$:
\[
V((Z(T)) - V(Z(0)) = \int_0^T \big\langle \partial V(Z(t)), \dt Z(t)\big\rangle \;dt.
\]
For $Z(t)\notin G_F(\bar{x})$, we have
\[
\partial V(Z(t)) = \frac{Z(t) - \proj_{G_F(\bar{x})}(Z(t))}{ \| Z(t) - \proj_{G_F(\bar{x})}(Z(t))\|}
\]
 and $\dt Z(t) = a (d(t) - Z(t))$ with some $d(t)\in G_1(\bar{x})$. Therefore,
\[
\big\langle \partial V(Z(t)), \dt Z(t)\big\rangle \le  - a \| Z(t) - \proj_{G_1(\bar{x})}(Z(t))\| = -a V(Z(t)).
\]
It follows that
\[
V((Z(T)) - V(Z(0)) \le - a \int_0^T  V(Z(t)) \;dt,
\]
and thus
\begin{equation}
\label{Z-conv}
\lim_{t\to\infty} \dist\big(Z(t),G_F(\bar{x})\big) = 0.
\end{equation}
It follows from \eqref{non-opt}--\eqref{Z-conv} that $T>0$ exists,
such that $ -Z(T) \not \in N_X(\bar{x})$, which yields $\dt X(T)\ne 0$. Consequently,
the path $X(t)$ starting from $\bar{x}$ cannot be constant (our supposition made right after \eqref{non-opt} cannot be true).  But if is not constant, then again $T>0$ exists, such that $\dt X(T)\ne 0$. By Step 1,
the triple $(X(T),Z(T),U(T))$ would have to be an accumulation point of the sequence $\{(x^k,z^k,u^k)\}$, a case already excluded.
We conclude that every accumulation point $(\bar{x},\bar{z},\bar{u})$  of the sequence $\{(x^k,z^k,u^k)\}$ is in $\Sc$.
The convergence of the sequence $\big\{W(x^k,z^k,u^k)\big\}$ then follows in the same way as
\citep[Thm. 3.20]{duchi2018stochastic} or \citep[Thm. 3.5]{majewski2018analysis}. As $\eta(x^k,z^k)\to 0$,
the convergence of $\{f(x^k,u^k)\}$ follows as well. Since $h(x^k)-u^k \to 0$, the sequence $\{F(x^k)\}$ is convergent
as well.
\end{proof}

\section{Numerical experiments}

In this section, we report results of numerical experiments that illustrate the performance of our single time-scale (STS) method for deep learning and logistic regression. For both applications, we consider perturbations in the training data set which leads to a distributional shift in the population measure $\mathbb{P}$, whereas we do not perturb the test data. We run the STS algorithm on the contaminated training data and investigate the robustness of the solution found by STS by considering different samples from the test data and the corresponding distribution of the test loss. {Our numerical results were obtained using Python (Version 3.7) on an Alienware Aurora R8 desktop with a 3.60 GHz CPU (i7-2677M) and 16GB memory.} 


\subsection{Deep learning} We consider a fully-connected network on two benchmark datasets: MNIST \citep{lecun2010mnist} and CIFAR10 \citep{Krizhevsky09learningmultiple}, where the model has the depth (the number of layers) of 3 and the width (the number of neurons per hidden layer) of 100. The MNIST dataset is split into a training dataset of $60000$ examples and a test dataset of $10000$ examples, whereas the CIFAR10 dataset is split into a training part of $50000$ examples and a test part of $10000$ examples. 
In both MNIST and CIFAR10 datasets, the output variable $y$ to be predicted is an integer valued from $0$ to $9$. We distort the distributions of MNIST and CIFAR10 training datasets by deleting all the data points with a $y$ value equal to $0$ (such points account for approximately 10$\%$ of the whole dataset). Based on the contaminated data, we train our model with different robustness levels $\kappa$ for 4000 iterations. To test the robustness of the model found by STS, we sample 100 points from the test dataset and compute the corresponding loss; and repeat this procedure 200 times for both datasets to generate a histogram of the test loss. We then report the corresponding cumulative distribution function (CDF) of the test loss in Figures \ref{figure-1} and  \ref{figure-2} for different values of $\kappa$, compared with results from a model trained by SGD.\footnote{{There are also adversarial learning methods \citep{madry2017towards,goodfellow2014explaining,kurakin2016adversarial,zhang2019you} where the aim is to be resistant to norm-bounded perturbations of the input before we have access to it; however, we do not compare with these methods as our formulation \eqref{stoch_prob} focuses on a distributional shift.}} 

If the training data are not contaminated at all, {we have observed in our experiments} that STS generates a similar or slightly worse solution than SGD. 
{This is expected as STS optimizes a penalized (robust) loss \eqref{stoch_prob} which is different than the empirical loss. The numerical details are omitted for the sake of brevity. 
On the other hand}, when the data contains distributional shifts, we see a clear advantage of the STS method over the SGD method.



\begin{figure}[h!]
\centering
\begin{subfigure}{0.43\textwidth}
  \includegraphics[width=\linewidth]{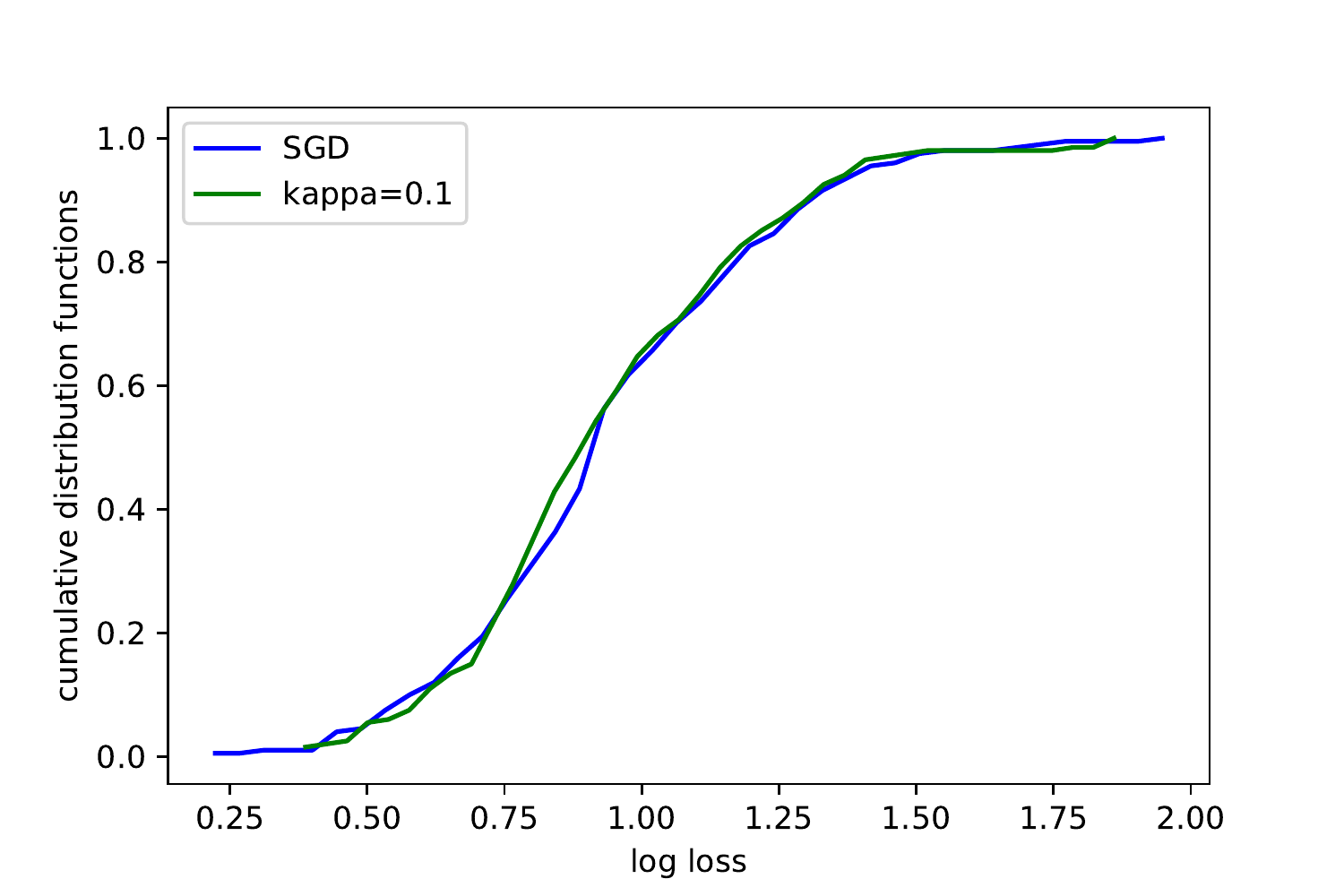}
  \caption{$\varkappa = 0.1$.}
\end{subfigure}
\begin{subfigure}{0.43\textwidth}
  \includegraphics[width=\linewidth]{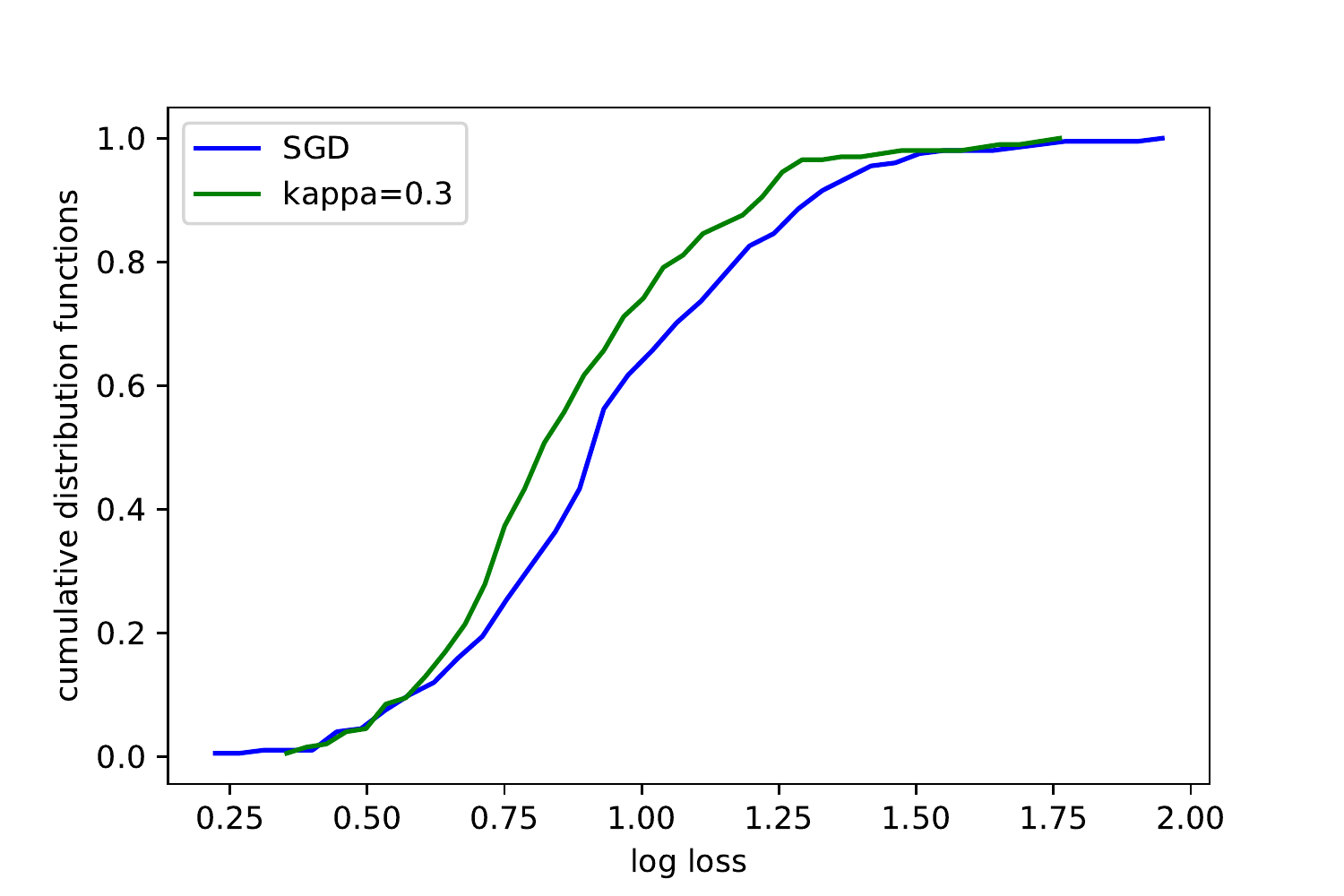}
  \caption{$\varkappa = 0.3$.}
\end{subfigure}
\begin{subfigure}{0.43\textwidth}
  \includegraphics[width=\linewidth]{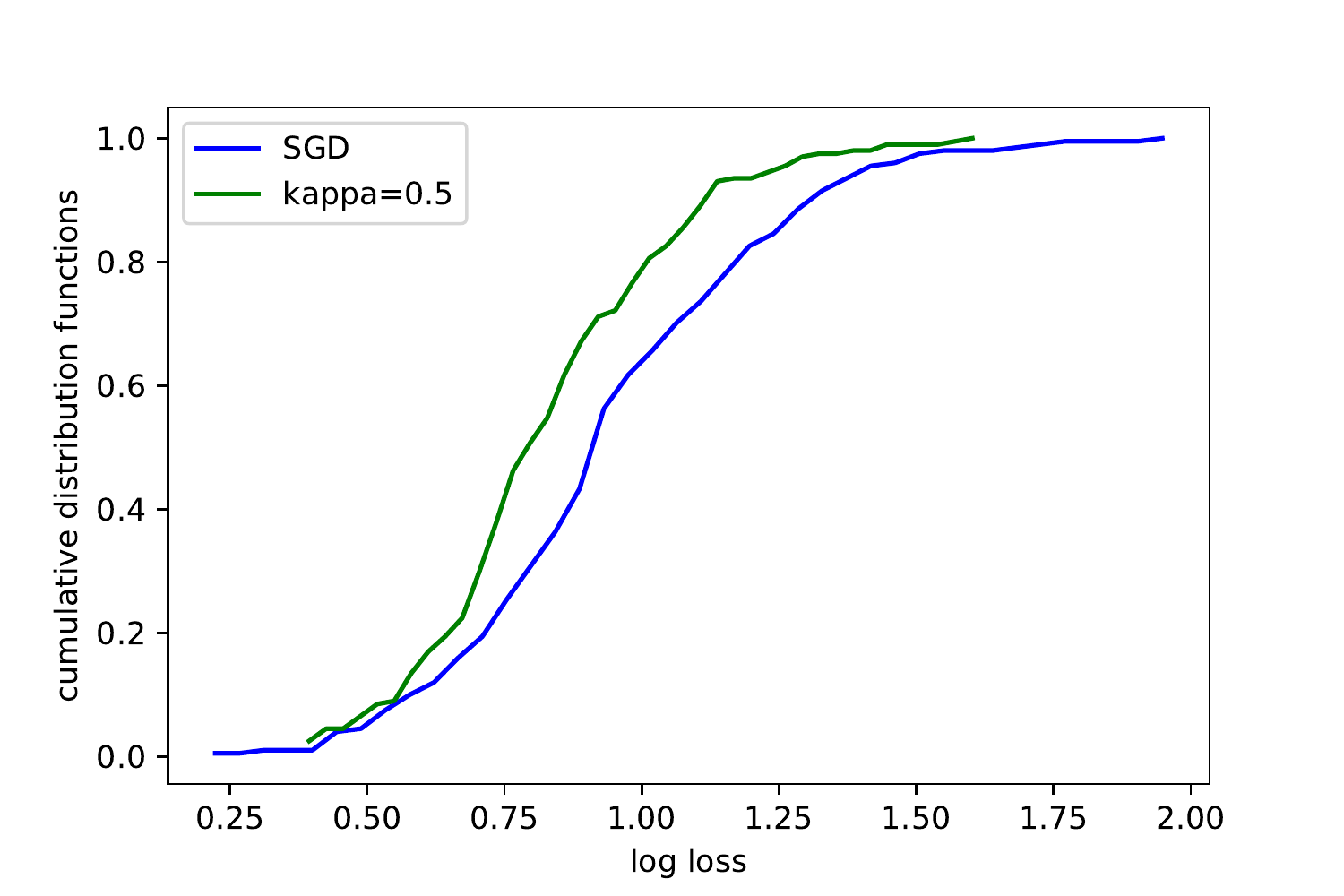}
  \caption{$\varkappa = 0.5$.}
\end{subfigure}
\begin{subfigure}{0.43\textwidth}
  \includegraphics[width=\linewidth]{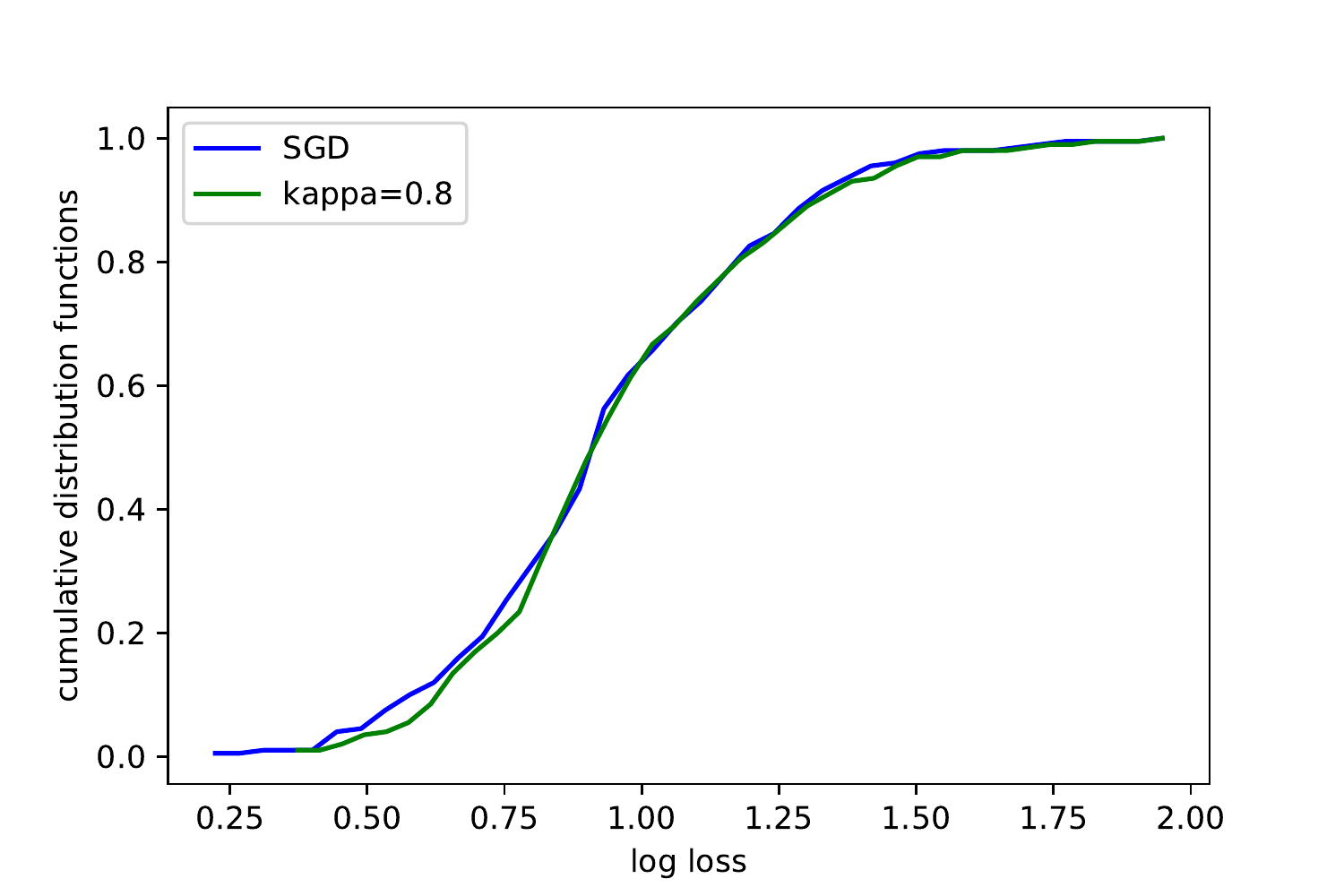}
  \caption{$\varkappa = 0.8$.}
\end{subfigure}
\caption{The CDFs of the SGD solution and the STS solutions under different robustness levels $\kappa$
for MNIST after 4000 iterations.}
\label{figure-1}
\end{figure}



\begin{figure}[h!]
\centering
\begin{subfigure}{0.43\textwidth}
  \includegraphics[width=\linewidth]{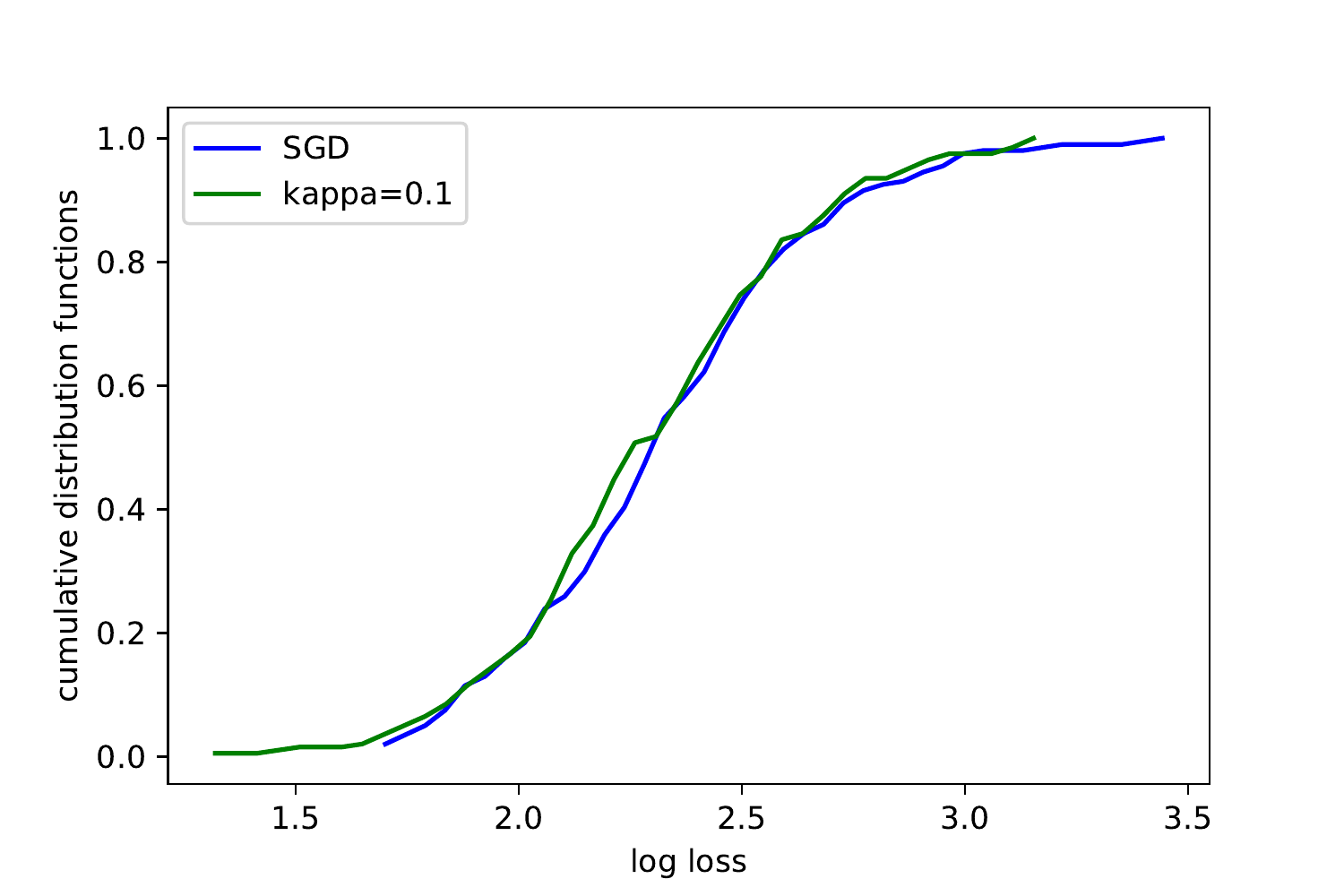}
  \caption{$\varkappa = 0.1$.}
\end{subfigure}
\begin{subfigure}{0.43\textwidth}
  \includegraphics[width=\linewidth]{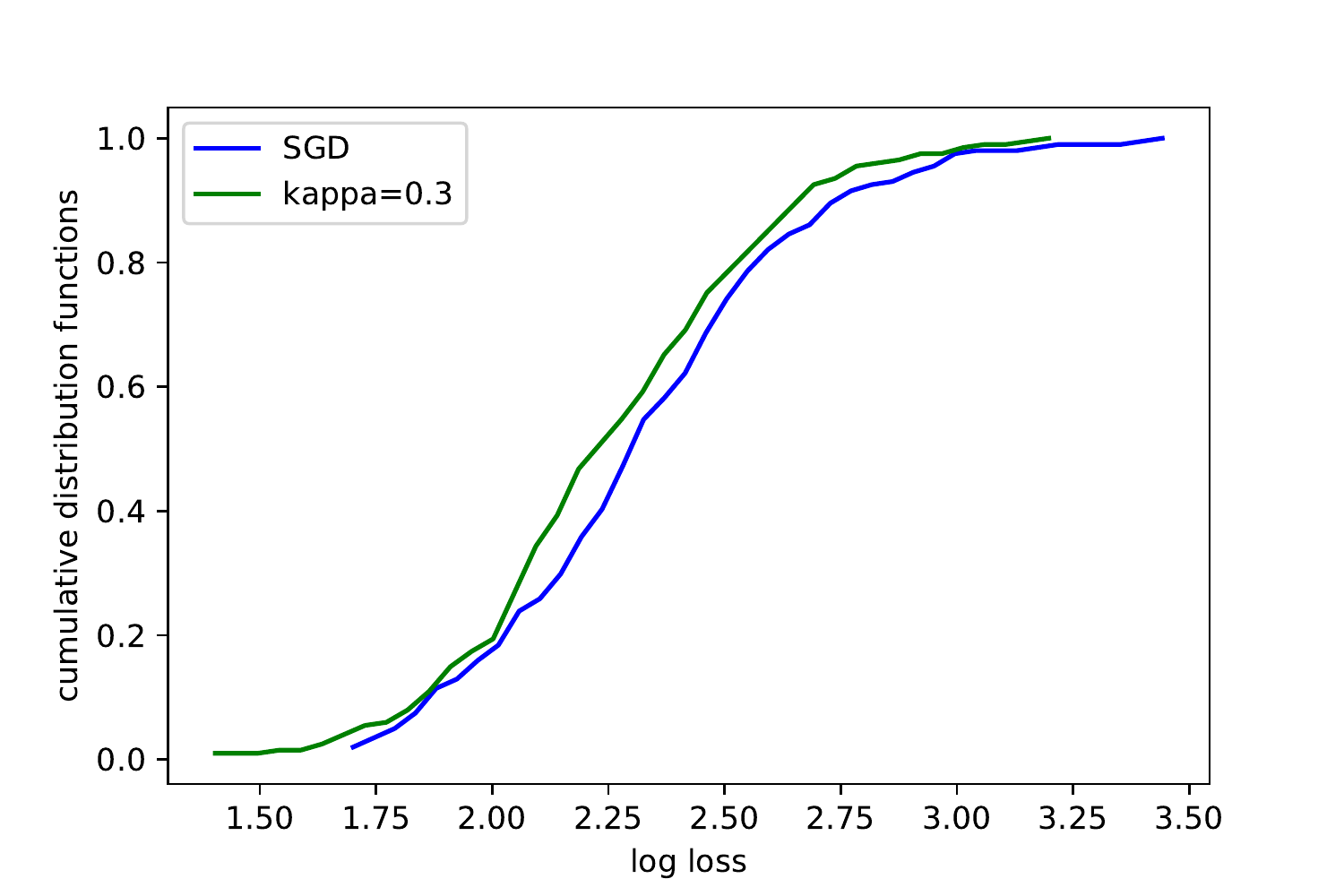}
  \caption{$\varkappa = 0.3$.}
\end{subfigure}
\begin{subfigure}{0.43\textwidth}
  \includegraphics[width=\linewidth]{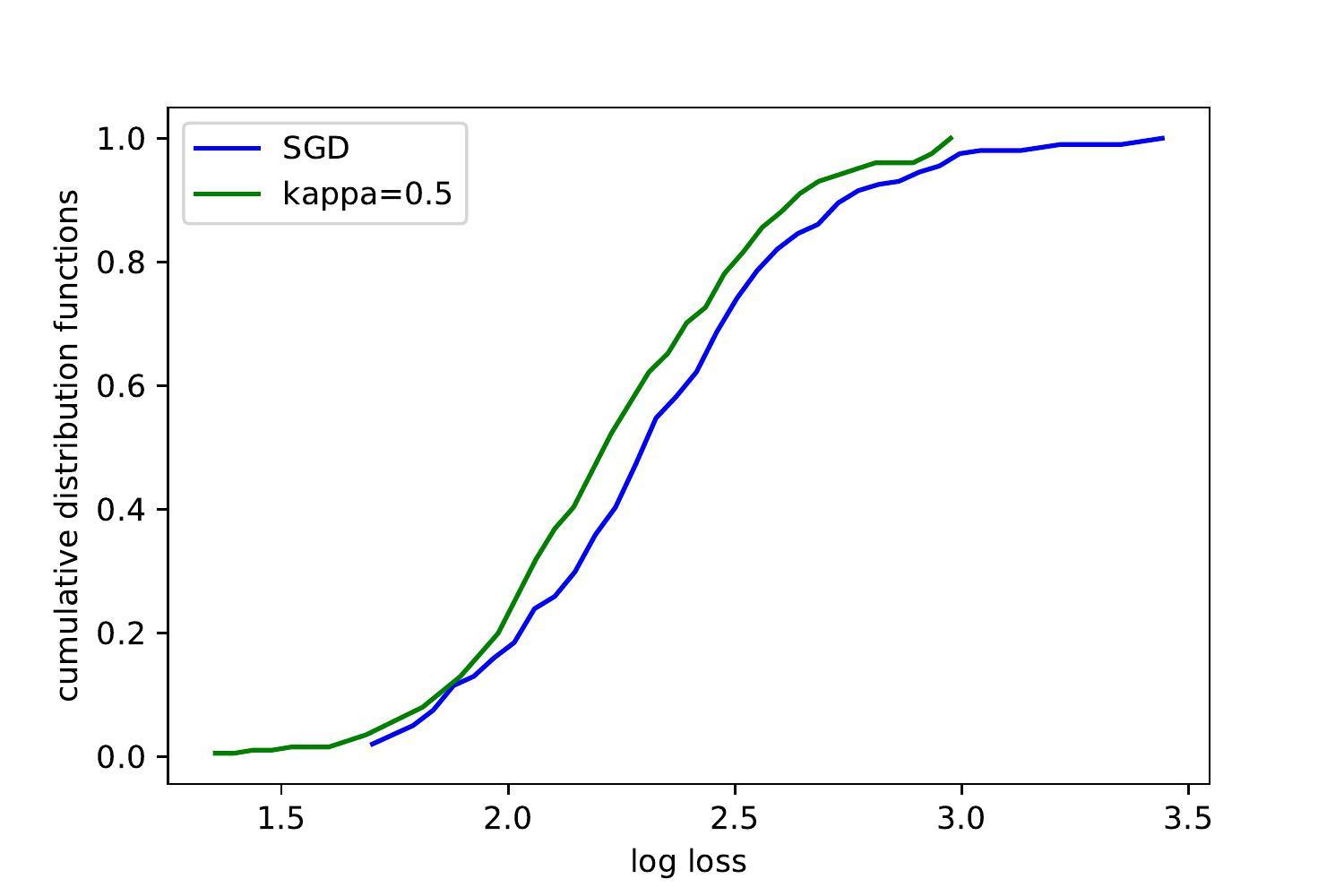}
  \caption{$\varkappa = 0.5$.}
\end{subfigure}
\begin{subfigure}{0.43\textwidth}
  \includegraphics[width=\linewidth]{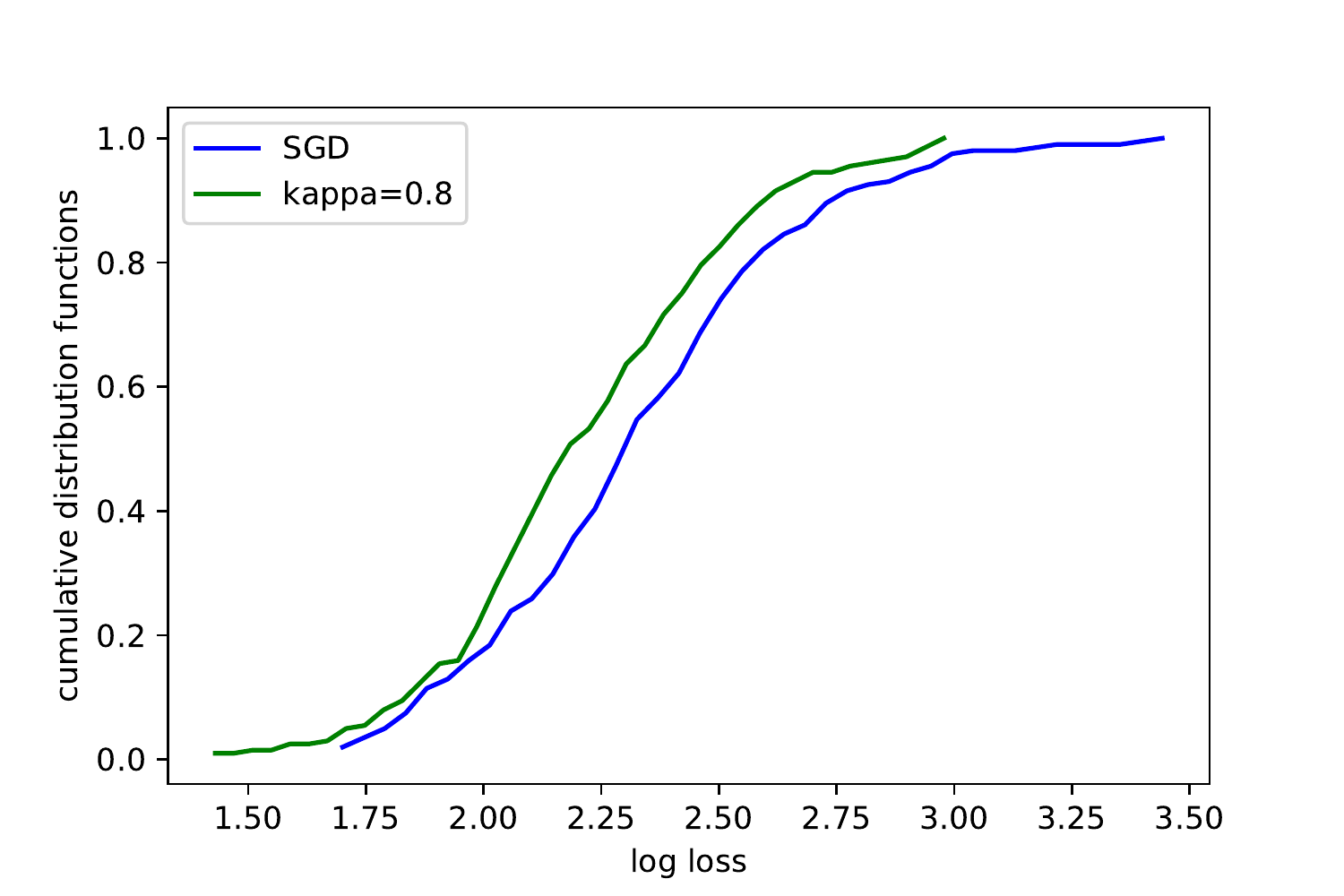}
  \caption{$\varkappa = 0.8$.}
\end{subfigure}
\caption{The CDFs of the SGD solution and the STS solutions under different robustness levels $\varkappa$
for CIFAR10 after 4000 iterations.}
\label{figure-2}
\end{figure}
\subsection{Logistic regression}
We consider binary logistic regression on the Adult dataset \citep{Dua:2019} where the loss function has the form 
\mg{$\ell(x,D) = \left[\log(1+\exp(-b\, a^Tx))\right]$ where $D=(a,b)$ is the input data.} 
The problem is to predict whether the annual income of a person will be above \$50,000 or not, based on $n=123$ predictor variables. The dataset has 32561 training examples and 16281 test examples. We follow a similar methodology as before, where we distort the training data by deleting $80\%$ of the data points with the corresponding income below \$50,000. We trained our model with STS and another state-of-the-art method Bandit Mirror Descent (BMD) developed in \citep{namkoong2016stochastic}, allowing both methods to execute the same {numbers of iterations}, which corresponds to 80000 iterations of the STS method. We then compare the cdf of the loss of the trained models based on 3000 samples from the test data. The results are reported in Figure
\ref{fig:duchi_r} for different values of the robustness level $\kappa$. We see that STS results in smaller errors {and conclude that our method has desirable robustness properties with respect to perturbations in the input distribution.}

\begin{figure}[h!]
\centering
\begin{subfigure}{0.43\textwidth}
  \includegraphics[width=1\linewidth]{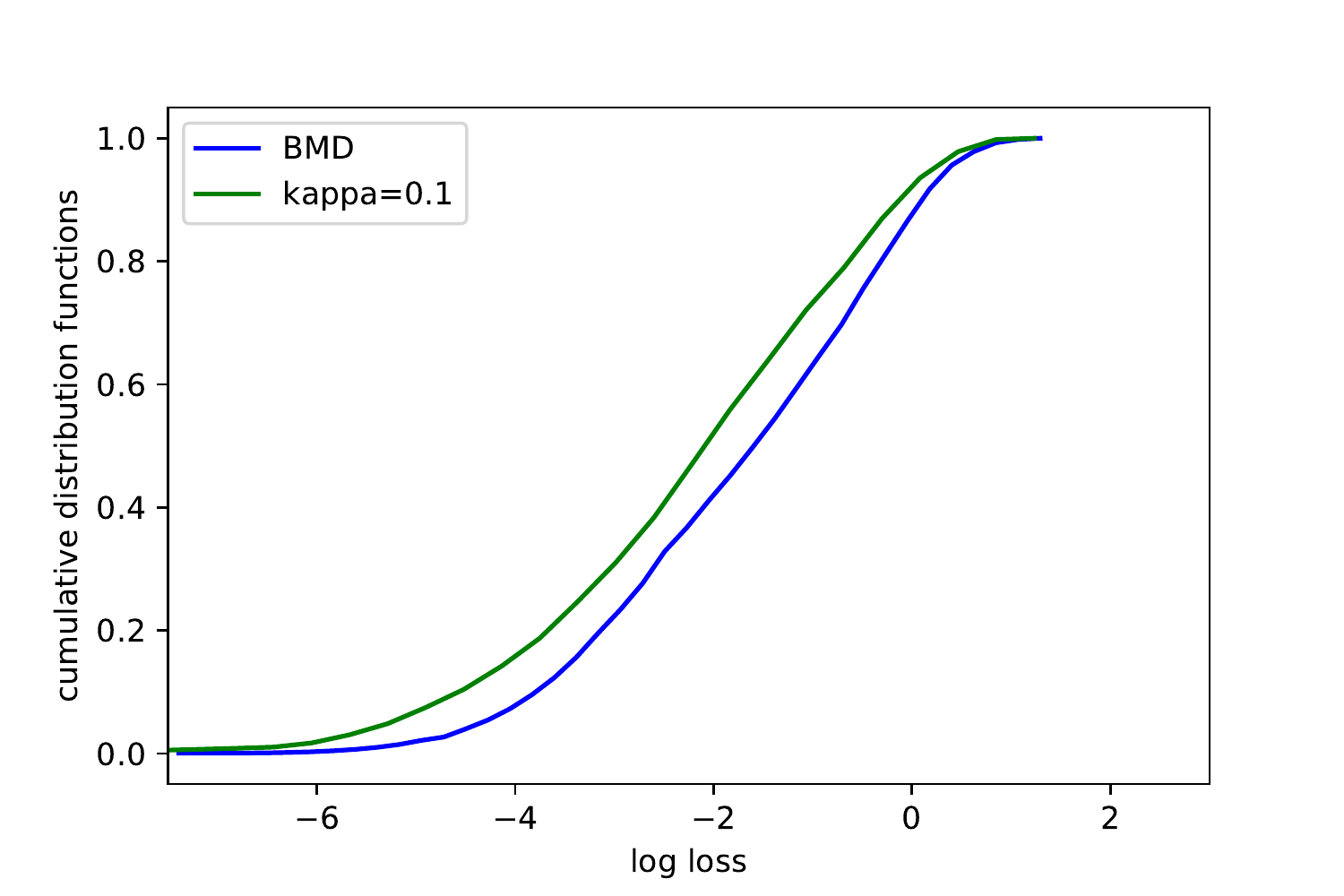}
  \caption{$\varkappa$ = 0.1}
\end{subfigure}
\begin{subfigure}{0.43\textwidth}
  \includegraphics[width=1\linewidth]{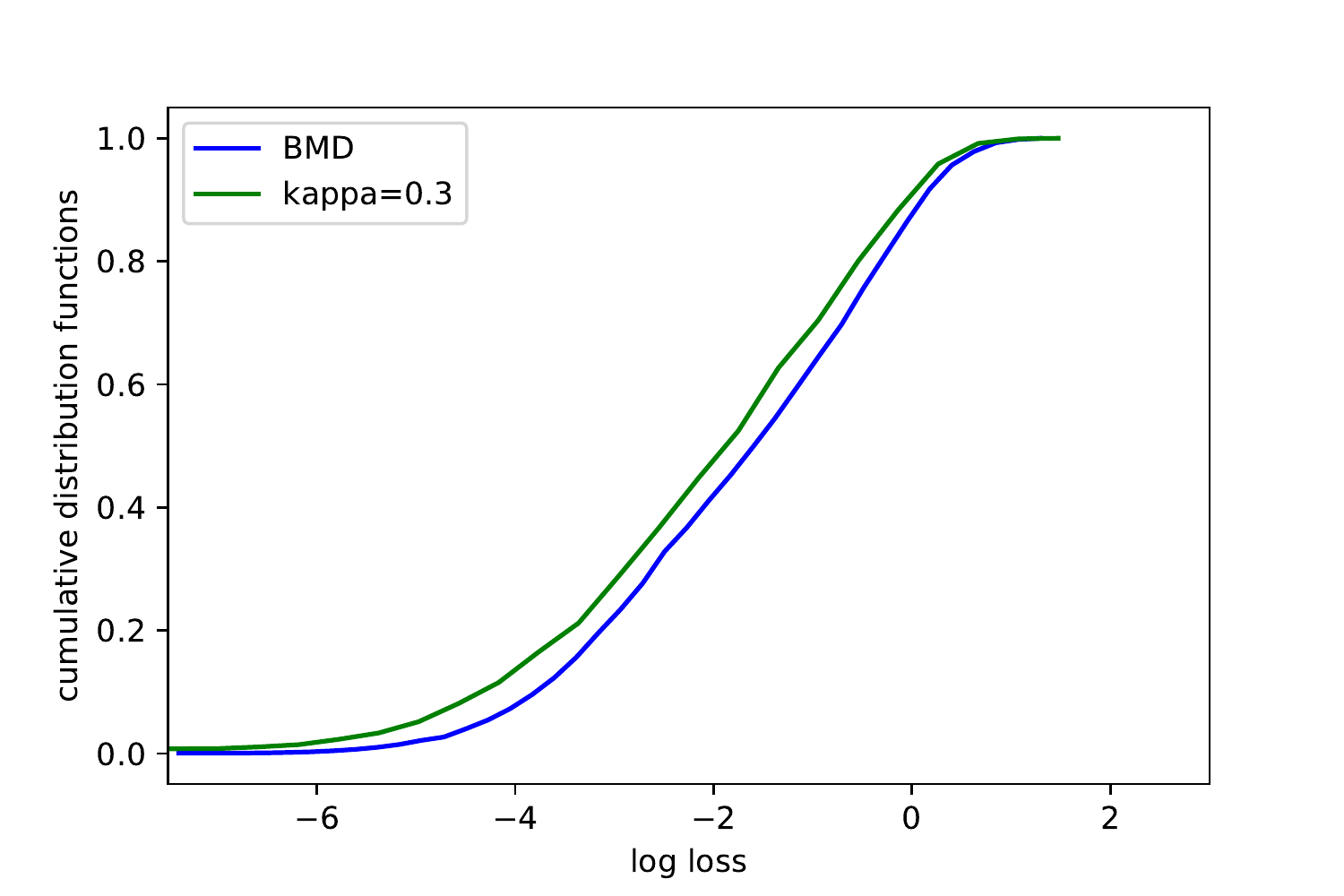}
  \caption{$\varkappa$ = 0.3}
\end{subfigure}
\begin{subfigure}{0.43\textwidth}
  \includegraphics[width=1\linewidth]{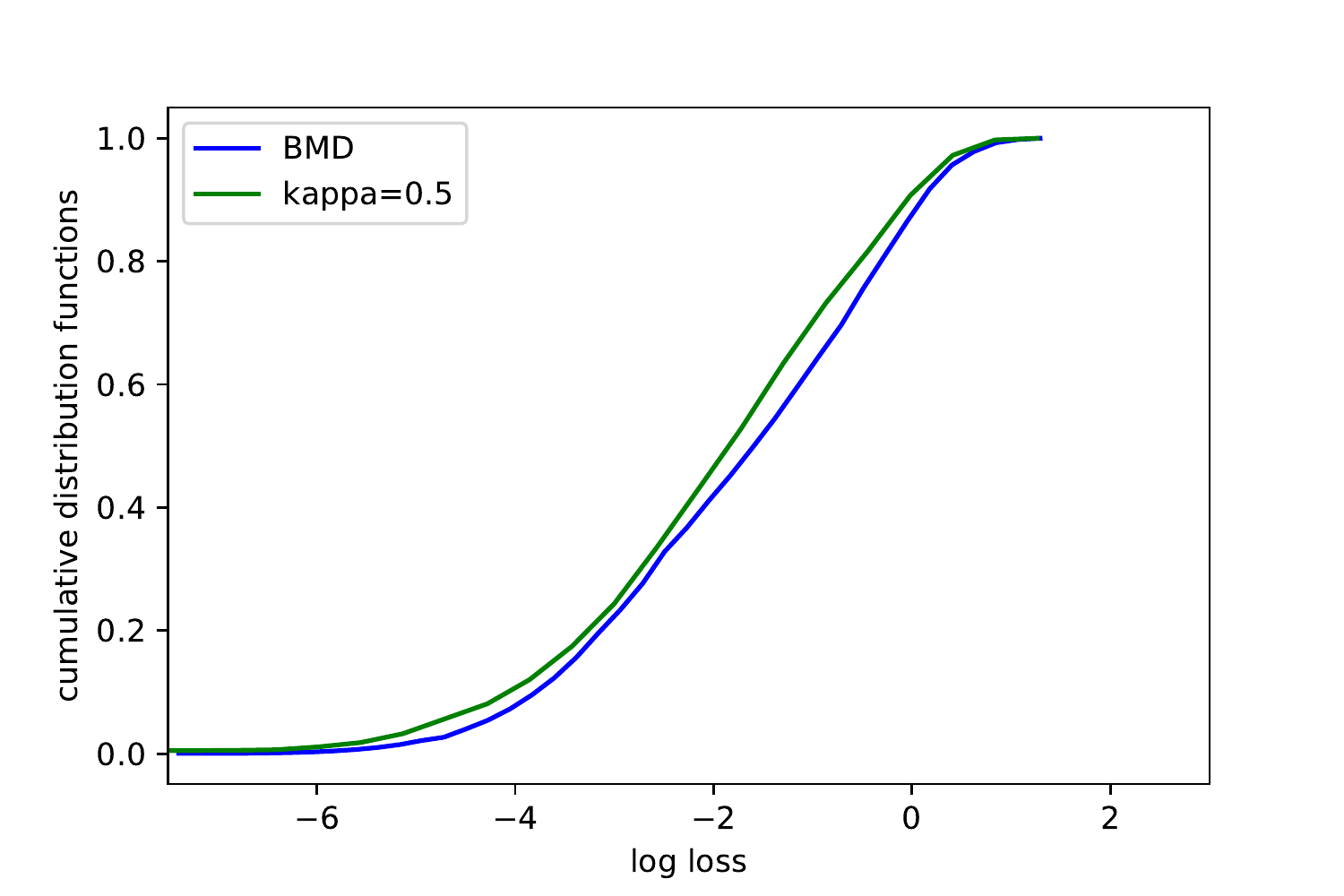}
  \caption{$\varkappa$ = 0.5}
\end{subfigure}
\begin{subfigure}{0.43\textwidth}
  \includegraphics[width=1\linewidth]{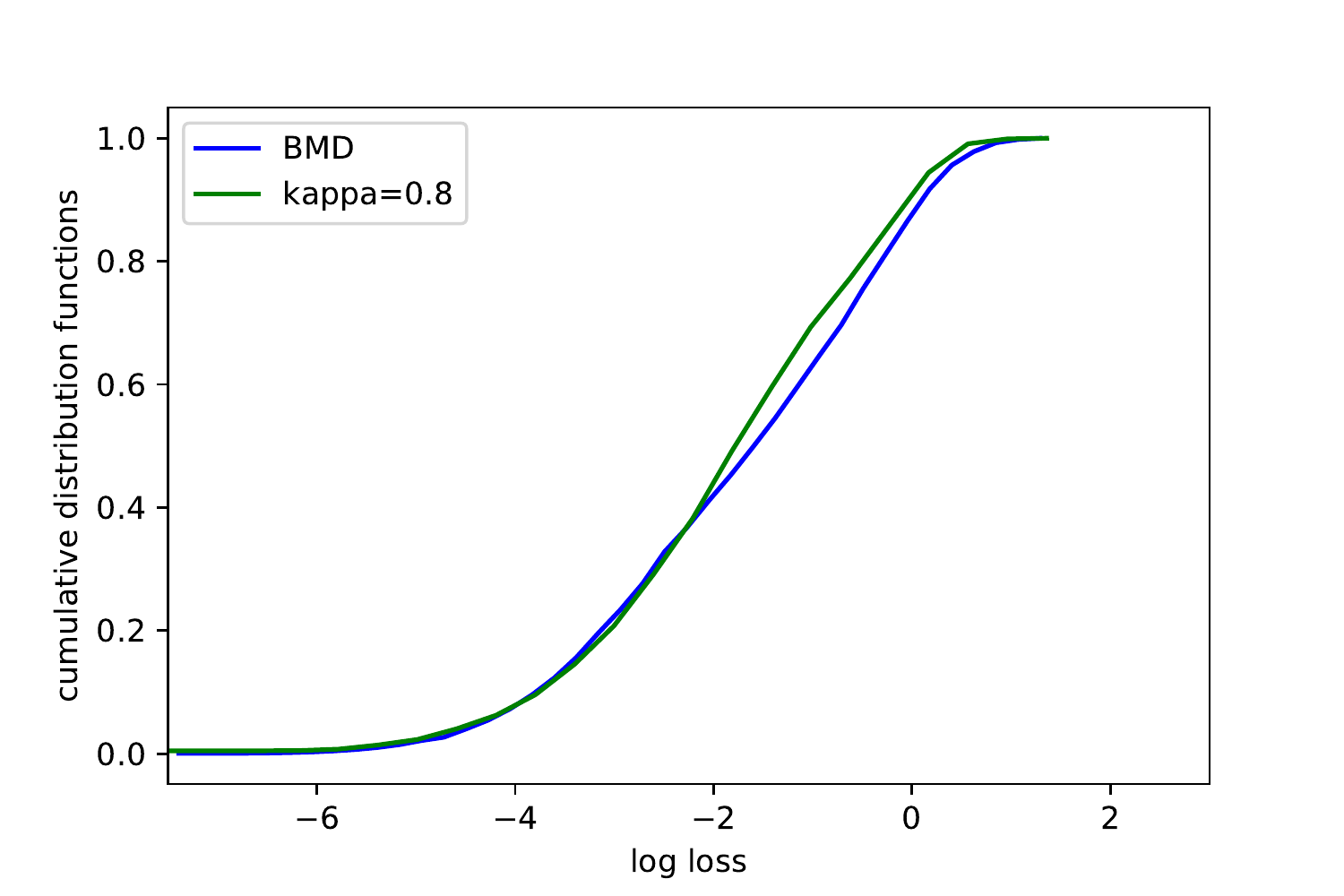}
  \caption{$\varkappa$ = 0.8}
\end{subfigure}
\caption{The CDFs of the BMD solution and the STS solutions under different robustness levels $\varkappa$ for the Adult dataset after 80000 iterations.}
\label{fig:duchi_r}
\end{figure}

\section{Acknowledgements}
Mert G\"urb\"uzbalaban's and Landi Zhu's research are supported in part by the grants Office of Naval Research Award Number N00014-21-1-2244, National Science Foundation (NSF) CCF-1814888, NSF DMS-2053485, NSF DMS-1723085.
\bibliography{multi}

\end{document}